\documentclass[10pt,oneside,a4paper,reqno]{amsart}  

\usepackage{amssymb,graphicx,mathrsfs,xcolor}
\usepackage[colorlinks=true, linkcolor=magenta, citecolor=cyan, urlcolor=blue]{hyperref}
\numberwithin{equation}{section}\newtheorem{theorem}{Theorem}[section]
\newtheorem{lemma}[theorem]{Lemma}

\newtheorem{proposition}[theorem]{Proposition}\theoremstyle{remark}
\newtheorem{remark}{Remark}[section]
\theoremstyle{definition}

\newcommand{\bra}[1]{\langle #1 \rangle}

\title[critical Dirac equation]
{Global small solutions to the
critical\\ radial Dirac equation with potential}
\date{\today}    

\author{Federico Cacciafesta}
\address{Federico Cacciafesta: SAPIENZA --- Universit\`a di Roma,
Dipartimento di Matematica, Piazzale A.~Moro 2, I-00185 Roma, Italy}
\email{cacciafe@mat.uniroma1.it}


\begin{document}\maketitle\begin{abstract}
  We solve globally a radial cubic Dirac equation perturbed
  with a small potential, with data of small critical norm
  $H^{1}$. The main tool are new endpoint estimates of
  the perturbed Dirac flow for a class of radial-type initial data.
\end{abstract}

\section{Introduction}\label{sec:introduction} 

Consider a zero mass nonlinear Dirac equation
\begin{equation}\label{pb1}
iu_t-\mathcal{D}u=F(u),
\qquad
u(0,x)=f(x)
\end{equation}
for the spinor field $u:\mathbb{R}_t\times\mathbb{R}_x^3\rightarrow \mathbb{C}^4$,  where
$\mathcal{D}$ is the operator 
defined by 
$$
\mathcal{D}=i^{-1}
\displaystyle\sum_{k=1}^3\alpha_k\partial_k=-i(\alpha\cdot\nabla)
$$ 
while the $4\times 4$ Dirac matrices are defined as
\begin{equation}
\alpha_k=\left(\begin{array}{cc}0 & \sigma_k \\\sigma_k & 0\end{array}\right),\quad k=1,2,3
\end{equation}
in terms of the Pauli matrices
\begin{equation}
\sigma_1=\left(\begin{array}{cc}0 & 1 \\1 & 0\end{array}\right),\quad
\sigma_2=\left(\begin{array}{cc}0 &-i \\i & 0\end{array}\right),\quad
\sigma_3=\left(\begin{array}{cc}1 & 0\\0 & -1\end{array}\right).
\end{equation}
We shall assume that the nonlinear term $F(u)$ is cubic of 
a very specific form, namely 
\begin{equation}\label{eq:NLFu}
  \text{either}\quad
  F(u)=\langle \beta u,u\rangle u
  \quad\textrm{or}\quad 
  F(u)=\langle u,u\rangle u,
\end{equation}
where $\langle\cdot,\cdot\rangle$ denotes the standard hermitian product in $\mathbb{C}^4$ and $\beta$ is the $4\times 4$ Dirac matrix
$$
\beta=\left(\begin{array}{cc}\mathbb{I}_2 & 0 \\0 & -\mathbb{I}_2\end{array}\right).
$$
Notice that nonlinearities of the form \eqref{eq:NLFu} are the
most interesting from a physical point of view
(see e.g.~\cite{Ranada83-a}).

The typical approach to the Dirac equation is based on the
identity
\begin{equation}\label{eq:waveid}
  (i \partial_{t}-\mathcal{D})(i \partial_{t}+\mathcal{D})\equiv
  (-\partial_{t}+\Delta)\mathbb{I}_{4}
\end{equation}
which follows from the anticommuting relations
\begin{equation}\label{comm1}
\alpha_i\alpha_k+\alpha_k\alpha_i=2\delta_{ik}\mathbb{I}_4,\quad i,k=1,2,3,
\end{equation}
\begin{equation}\label{comm2}
\alpha_i\beta+\beta\alpha_i=0,\quad i=1,2,3,\qquad
\beta^2=\mathbb{I}_4.
\end{equation}
In many cases the identity allows to reduce problems 
for the massless Dirac equation
to analogous ones for the wave equation, for which many effective
tools are available.

The classical strategy to solve
nonlinear dispersive equations is by the use of a fixed point argument in
a suitable space, via the appropriate space-time 
\emph{Strichartz estimates}.
A huge literature is available on these estimates for several 
dispersive operators (wave, Schroedinger, Klein-Gordon, Dirac 
and others). For homogeneous nonlinear terms one typically
finds a threshold regularity $s_{c}$ in the scale of Sobolev
spaces $H^{s}$, such that for
\emph{subcritical data} with $s>s_{c}$ solvability holds,
while for \emph{supercritical data} with $s<s_{c}$ one has
various degrees of ill-posedness. Data of critical regularity
give rise to difficult questions which depend on the precise
structure of the equation.

If we denote with $e^{it\mathcal{D}}$ the propagator
for the Dirac operator $\mathcal{D}$, which can be defined via
Fourier transform as the operator of symbol $e^{it \xi \cdot \alpha}$,
Strichartz estimates for the linear 3D Dirac equation can be written
as follows 
(see \cite{GinibreVelo95-b},
\cite{KeelTao98-a}, and,
for the Dirac equation,
\cite{DanconaFanelli07-a},
\cite{DanconaFanelli08-a}):
\begin{equation}\label{freestr}
\| |D|^{\frac{1}{q}-\frac{1}{p}-\frac{1}{2}}
e^{it\mathcal{D}}f\|_{L^p_t L^{q}_x}\lesssim
\| f\|_{L^2}
\end{equation}
where the exponents $(p,q)$ are \emph{admissible}, 
that is to say
$$
\frac{2}{p}+\frac{2}{q}=1,\quad 2< p\leq\infty,\quad2\leq q<\infty.
$$
Here and in the following we shall use the notation
$|D|=(-\Delta)^{1/2}$, the mixed space-time Strichartz space
$L^p_tL^q_x$ is
$L^p(\mathbb{R}_t; L^q(\mathbb{R}^n_x))$,
and $\dot H^s$ is the homogeneous Sobolev space with  norm 
$\| f\|_{\dot H^s}=\| |D|^sf\|_{L^2}$.
As it is well known, estimate \eqref{freestr} is false for
the \emph{endpoint} couple $(p,q)=(2,\infty)$, which would correspond
to the estimate
\begin{equation}\label{end}
\| e^{it\mathcal{D}}f\|_{L^2_t L^\infty_x}\lesssim
\| f\|_{\dot H^1}.
\end{equation}
The failure of \eqref{end} for generic data
was first noticed for 3D the wave equation in
\cite{KlainermanMachedon93-a},
and the corresponding statement for
the Dirac equation follows easily from the relation connecting
the Dirac and wave flows
(see \eqref{dirsol} below).

A clever use of the estimates \eqref{freestr} is sufficient for
the study of \eqref{pb1} in the subcritical case.
Local existence, and global existence for small data, was proved
in \cite{EscobedoVega97-a} for nonlinearities of the form
$\langle \beta u,u\rangle^{(p-1)/2} u$ with $p>3$ and data
in $H^{s}$ with $s>3/2-1/(p-1)$.
The case of cubic nonlinearities and small data in
$H^{s}$, $s>1$, was solved in \cite{MachiharaNakamuraOzawa04-a},
where the result was generalized to all space dimensions.
We also mention that systems involving the Dirac equation,
like Dirac-Klein-Gordon and Maxwell-Dirac, have beed the
object of intense attention and the subcritical theory was
completed only recently
(see
\cite{DanconaFoschiSelberg07-a},
\cite{DanconaFoschiSelberg07-b},
\cite{DanconaFoschiSelberg10-a}).

The failure of the endpoint estimate \eqref{freestr}
means that above methods break down in the critical case
of a cubic nonlinearity with $H^{1}$ data, and indeed the problem
of global existence is still open in this case.
In an attempt to overcome this limitation, in
\cite{MachiharaNakamuraNakanishi05-a}
the following precised estimate was proved:
\begin{equation}\label{giapp}
\| e^{it\mathcal{D}}f\|_{L^2_tL^\infty_{|x|}L^p_\theta}\lesssim
\| f\|_{\dot{H}^1}\qquad
\forall p<\infty.
\end{equation}
The norm at the left hand side distinguishes between the
integrability in the radial and tangential
directions. Using estimate \eqref{giapp}, Machihara et al. were able
to prove global well posedness for problem (\ref{pb1}) for small $\dot{H}^1$-norm data with slight additional angular regularity,
and in particular for all \emph{radial} $\dot H^{1}$ data. This
is especially interesting since, as we shall see, radial data
do not correspond to radial solution for the Dirac equation.

The main goal of the present paper is to study a special
class of solutions to equation \eqref{pb1} and more
generally to a potential perturbation of the form
\begin{equation}\label{pb2}
iu_t-\mathcal{D}u+V(x) u=F(u),
\qquad
u(0)=f(x),
\end{equation}
which can be considered as a suitable generalization of the
radial solutions to the wave equation in the context of the
Dirac equation.
We recall that, for small potentials $V$ with suitable decay
at infinity, the full range of Strichartz estimates \eqref{freestr}
holds also for the perturbed flow
$e^{it(\mathcal{D}+V)}$, as proved in \cite{DanconaFanelli08-a}.
Hence subcritical problems can be treated exactly as in the
unperturbed case and one can extend the results of 
\cite{EscobedoVega97-a}, \cite{MachiharaNakamuraOzawa04-a}
to \eqref{pb2} in a straightforward way.
Here we focus on the more difficult case of
critical $H^{1}$ data with an additional symmetry
assumption like
\begin{equation*}%
  f=f_1+\mathcal{D}f_2
  \quad\text{with}\quad
  f_1\in\dot{H}^1,\;f_2\in\dot{H}^2,\quad f_1,f_2\;\rm{radial}.
\end{equation*}
In addition, in order to preserve the symmetry of solutions, 
we need to assume that the potential $V(x)$ is
\emph{spherically symmetric} in the sense of \cite{Thaller92-a}.

Our first result is an endpoint estimate for the linear flow:

\begin{theorem}\label{teo1}
Let $V(x)$ be a $4\times 4$ matrix of the form 
\begin{equation}\label{pppot}
V(x)=V_1(|x|)\mathbb{I}_4+i\beta(\alpha\cdot\hat{x})V_2(|x|),\qquad
V_1,V_2:\mathbb{R}^{+}\to \mathbb{R},\quad x\in \mathbb{R}^{3}
\end{equation}
(where $\widehat{x}=x/|x|$). Assume that
for some $\sigma>1$ and some sufficiently small
$\delta>0$
\begin{equation}\label{Vhp}
  |V(x)|\le \frac{\delta}
  {|x|^{1/2}|\log|x||^{\sigma/2}+|x|^{\sigma}}.
\end{equation}
Then the following endpoint Strichartz estimate
\begin{equation}\label{endV}
\| e^{it(\mathcal{D}+V)}f\|_{L^2_t L^\infty_x}\lesssim
\| f\|_{\dot H^1}
\end{equation}
holds for all initial data $f\in\mathcal{\dot H}^1$ where
\begin{equation}\label{Hclass}
\mathcal{\dot H}^1=
\{f_1+\mathcal{D}f_2,\;f_1\in\dot{H}^1(\mathbb{R}^{3}),
\;f_2\in\dot{H}^2(\mathbb{R}^{3}),
\;f_1,f_2\;
\rm{radial}\}.
\end{equation}
\end{theorem}

We recall that condition \eqref{Vhp} 
is sufficient to ensure that the perturbed Dirac operator 
$\mathcal{D}+V$ is self-adjoint on the domain $H^1(\mathbb{R}^3)^4$.
For this and many other properties of the Dirac equation,
a comprehensive reference is \cite{Thaller92-a}.

The natural application of estimate \eqref{endV} is to prove
global well posedness for the critical equation \eqref{pb2}.
However the nonlinear term $F(u)$ does not operate on the
space $\mathcal{\dot H}^{1}$ and 
additional restrictions on the algebraic
structure of the data are necessary.
More precisely, it is possible
to decompose the space $L^{2}(\mathbb{R}^{3})^{4}$
as a direct sum
\begin{equation*}
  L^{2}(\mathbb{R}^{3})^{4}\simeq
  \bigoplus_{j=\frac12,\frac32,\dots}^{\infty}
  \bigoplus_{m_{j}=-j}^{j}
  \bigoplus_{k_{j}=\atop\pm(j+1/2)}
  L^{2}(0,+\infty;dr)
  \otimes
  H_{m_{j},k_{j}}.
\end{equation*}
where the spaces $H_{m_{j},k_{j}}$ are two dimensional and
are generated by spherical harmonics on the sphere $\mathbb{S}^{2}$
(this is called a dcomposition in \emph{partial wave subspaces}).
When $j=1/2$, we have four spaces
\begin{equation*}
  L^{2}(0,+\infty;dr)
  \otimes
  H_{m_{1/2},k_{1/2}}
\end{equation*}
corresponding to the four possible choices of indices
\begin{equation}\label{eq:indexch}
  (m_{1/2},k_{1/2})=
  (-1/2,-1), \quad
  (-1/2,1), \quad
  (1/2,-1), \quad
  (1/2,1).
\end{equation}
Then we notice the mildly surprising fact that
each of these four spaces is invariant not only for the Dirac operator 
but also for the action of cubic nonlinearities
of the forms \eqref{eq:NLFu}. In a sense, these spaces
can be considered as a suitable 
generalization of radial functions adapted to the structure 
of the nonlinear problem
\eqref{pb2}. A detailed analysis of the partial wave decomposition is
given in Section \ref{thaller}, with explicit forms for the
functions in these spaces (see Lemma \ref{1/2}
and in particular \eqref{phi1}--\eqref{phi4}).

Thanks to this invariance, we can prove the following global
existence result:

\begin{theorem}\label{teo2}
Consider the equation on $\mathbb{R}\times \mathbb{R}^{3}$
\begin{equation}\label{nlpb}
iu_t-\mathcal{D}u+V(x)u=F(u),\qquad
u(0,x)=f(x)
\end{equation}
where the potential has the form
\begin{equation*}
  V=V_1(|x|)\mathbb{I}_4+i\beta(\alpha\cdot\hat{x})V_2(|x|),\qquad
  \widehat{x}=x/|x|
\end{equation*}
and satisfies assumption \eqref{Vhp}, while
\begin{equation*}
  F(u)=
  \langle \beta u,u\rangle u
  \qquad\text{or}\qquad
  F(u)=
  \langle u,u\rangle u.
\end{equation*}
Assume the initial data $f$ belong to a space
$\dot H^1((0,\infty),dr)\otimes  \mathcal{H}_{m_{1/2},k_{1/2}}$
for one of the choices \eqref{eq:indexch}.
Then if the $\dot H^{1}$ norm of the data is sufficiently small,
problem \eqref{nlpb} has a unique global solution
in the class
$C_t(\mathbb{R},\dot H^1)\cap L^2_t(\mathbb{R},L^\infty).$
\end{theorem}

The paper is organized as follows. Sections 2 contains an extension
of the endpoint estimate for the free Dirac operator which is then
adapted in Section 3 to a mixed endpoint-smoothing estimate with
weights for the nonhomogeneous linear Dirac equation.
Section 4 is devoted to the proof of Theorem (\ref{teo1}).
In Section 5 we recall the structure of the Dirac operator, the
partial wave decomposition, and we investigate the interaction of
the algebraic structure with the nonlinear term. Section 6 contains
the proof of Theorem (\ref{teo2}). 
\\\\The author wishes to thank prof. Piero D'Ancona for his help and the referee for some useful suggestions.

\section{The homogeneous endpoint estimate}\label{sec:omo}

Consider the 3-dimensional, massless, free Dirac 
equation (\ref{pb1}). As already observed,
the Dirac flow does not preserve radiality so that
we cannot hope to repeat here the simple argument used
in \cite{KlainermanMachedon93-a} for the wave equation.
However we can prove an endpoint estimate for suitable classes
of function; to this end we need a deeper insight in the 
structure of the Dirac operator expressed in radial coordinates.

Let $u=e^{it \mathcal{D}}f$
and notice that, thanks to identity \eqref{eq:waveid}, $u$ is
(formally) a solution of
\begin{equation}\label{wave}
\begin{cases}
\square u=0\\
u(0)=f(x)\\
u_t(0)=i\mathcal{D}f.
\end{cases}
\end{equation}
This gives the formula
\begin{equation}\label{dirsol}
\displaystyle u=e^{it\mathcal{D}}f=\cos(t|D|)f+i\frac{\sin(t|D|)}{|D|}\mathcal{D}f
\end{equation}
and we easily see that this representation is valid for generic
distribution data.

We start by proving the following result:

\begin{proposition}\label{prop1}
Let $f$ belong to the space $\mathcal{\dot{H}}^1$
defined in \eqref{Hclass}.
Then the following endpoint Strichartz estimate holds:
\begin{equation}\label{homdir}
\| e^{it\mathcal{D}}f\|_{L^2_tL^\infty_x}\lesssim
\| f\|_{\dot H^1}.
\end{equation}
\end{proposition}

\noindent
By formula \eqref{dirsol}, we see that the proof is an immediate 
consequence of the following Lemma. 
Notice that the proof of the first estimate \eqref{hom1} is inspired
by an argument of \cite{FangWang08-a} which holds for all $n\ge3$
without modification, while we fix $n=3$ in the second
estimate.

\begin{lemma}\label{lem1}
Let $f$ be a radial function and let $e^{it|D|}$ be the linear propagator associated to the wave operator. Then the following estimates hold:
\begin{equation}\label{hom1}
\| e^{it|D|}f\|_{L^2_tL^\infty_x}\lesssim
\| f\|_{\dot {H}^\frac{n-1}{2}}\qquad (n\ge3)
\end{equation}
\begin{equation}\label{hom2}
\left\| \frac{e^{it|D|}}{|D|}\mathcal{D}f
\right\|_{L^2_tL^\infty_x}\lesssim
\| f\|_{\dot {H}^1}\qquad (n=3).
\end{equation}
\end{lemma}

\begin{proof}
Using Fourier transform in spherical coordinates and the radiality of $f$ 
and setting $\rho=|x|$, $|\xi|=\lambda$, 
$x\cdot\xi=\rho\lambda\cos\theta$, we have
\begin{equation}\label{radialw}
e^{it|D|}f=\int e^{i(x\cdot\xi+t|\xi|)}\hat f(\xi)d\xi=
\\\int_0^\infty\int_0^\pi e^{i\lambda(t+\rho\cos\theta)}\hat f(\lambda)\lambda^{n-1}(\sin\theta)^{n-2} d\theta d\lambda.
\end{equation}
With the change of variable $y=\cos\theta$ in (\ref{radialw}) we obtain
$$
=\displaystyle\int_{\mathbb{R}}d\lambda\: e^{it\lambda}g(\lambda)\int_{-1}^1e^{i\lambda\rho y}(1-y^2)^\frac{n-3}{2}dy
$$
with $g(\lambda)=\hat f(\lambda)\lambda^{n-1}H(\lambda)$ ($H$ represents the classical Heaviside function).
Now changing the order of the integrals we obtain
$$
=\displaystyle\int_{-1}^1dy\:(1-y^2)^\frac{n-3}{2}\int_{-\infty}^{+\infty}e^{i\rho(t+\lambda y)}g(\rho)d\rho=
\displaystyle\int_{-1}^1dy\:(1-y^2)^\frac{n-3}{2}\hat g(t+\lambda y)
$$
Since for $n\geq3$ one has
$\displaystyle(1-y^2)^\frac{n-3}{2}\leq1$, the change of variable $y\rightarrow y/r$ yields again\begin{equation}\label{max}
\displaystyle\leq
\frac{1}{r}\int_{-r}^r\hat g(t+y)dy=M(\hat g)(t)
\end{equation}
where $M$ denotes the standard maximal operator.
Then we have for all $t$
\begin{equation}\label{bound1}
\| e^{it|D|}f\|_{L^\infty_x}\lesssim M(\hat g)(t),
\end{equation}
and thus by the $L^p$-boundedness of maximal operator and Plancherel's theorem
\begin{equation}
\displaystyle\| e^{it|D|}f\|_{L^2_t L^\infty_x}\lesssim \| g\|_{L^2_t}=
\left(\int_0^\infty(\lambda^{n-1}|\hat f|)^2d\lambda\right)^\frac{1}{2}=
\left(\int_0^\infty\left|\lambda^{\frac{n-1}{2}}|\hat f|\right|^2\lambda^{n-1}d\lambda\right)^\frac{1}{2}=
\end{equation}
$$
=\| \lambda^{\frac{n-1}{2}}\hat f\|_{L^2}=\| f\|_{\dot{H}^\frac{n-1}{2}}
$$
which gives (\ref{hom1}). 

We now turn to estimate (\ref{hom2}), for which the calculations are 
similar. 
Indeed we can write (here we are fixing $n=3$)
\begin{equation}\label{repDF}
\displaystyle\frac{e^{it|D|}}{|D|}\mathcal{D}f=\int e^{i(x\cdot\xi+t|\xi|)}
(\alpha\cdot\hat\xi)\hat f(\xi)d\xi=
\end{equation}
where
\begin{equation*}
  \alpha\cdot\hat\xi=\sum_{k=1}^3\frac{(\alpha_k\cdot\xi_k)}{|\xi|}.
\end{equation*}
Using spherical coordinates as before we have
\begin{equation}\label{dirrad}
\displaystyle\int_0^\infty d\lambda\int_0^{2\pi}d\phi\int_0^\pi d\theta
e^{i\lambda(t+\rho\cos\theta)}A(\theta,\phi)\hat f(\lambda)
\lambda^2\sin\theta
\end{equation}
with the operator $A(\theta,\phi)=\alpha_1\cos\theta+\alpha_2\sin\theta\cos\phi+\alpha_3
\sin\theta\sin\phi$.
Observing that 
$$
\displaystyle\int_0^{2\pi}\alpha_2\sin\theta\cos\phi\:d\phi=
\displaystyle\int_0^{2\pi}\alpha_3\sin\theta\sin\phi\:d\phi=0,
$$
we see that (\ref{dirrad}) is equal to
$$
\cong\displaystyle\int_0^\infty d\lambda\int_0^\pi d\theta
e^{i\lambda(t+\rho\cos\theta)}\alpha_1\cos\theta\hat f(\lambda)\lambda^2\sin\theta.
$$
Setting as before $g(\lambda)=\hat f(\lambda)\lambda^2H(\lambda)$, changing variable $\cos\theta\rightarrow y$ and then $y\rightarrow y/r$ yield
\begin{equation}\label{max2}
=\displaystyle\int_{-1}^1dy\left(\alpha_1\cdot y\right)\int_{-\infty}^{+\infty}d\rho \:e^{i\lambda(t+\rho y)}g(\rho)=
\frac{1}{r}\int_{-r}^r\left(\alpha_1\cdot\frac{y}{r}\right)\hat g(t+y)\:dy
\cong
c M(\hat g)(t)
\end{equation}
since the term $\displaystyle\left(\alpha_1\cdot\frac{y}{r}\right)$ is bounded, and so we have the bound
\begin{equation}\label{bound2}
\left\| \frac{e^{it|D|}}{|D|}\mathcal{D}f\right\|_{L^\infty_x}
\lesssim M(\hat g)(t).
\end{equation}
The
$L^p$-boundedness of maximal operator and Placherel Theorem yield as above 
\begin{equation}
\left\|\frac{e^{it|D|}}{|D|}\mathcal{D}f\right\|_{L^2_tL^\infty_x}
\lesssim\| f\|_{\dot H^1}
\end{equation}
which gives the desired estimate (\ref{hom2}).
\end{proof}

Combining estimates (\ref{hom1}) and (\ref{hom2}) and using representation (\ref{dirsol}) for the solution of the free Dirac system 
we obtain estimate (\ref{homdir}).

\section{The mixed endpoint-smoothing estimate}\label{sec:omo}

We consider now the non homogeneous equation
\begin{equation}\label{dirnono}
iu_t-\mathcal{D}u=F(t,x),\qquad
u(0,x)=0.
\end{equation}
By Duhamel's formula and the
representation (\ref{dirsol}) we can write the solution $u$ as
\begin{equation}\label{nonomrep}
\displaystyle u(t,x)=\int_0^t e^{i(t-s)\mathcal{D}}F(s,x)ds=
\end{equation}
$$
=\int^t_0\left(\cos((t-s)|D|)F(s,x)+
i\frac{\sin((t-s)|D|)}{|D|}\mathcal{D}F(s,x)
\right)ds.
$$
Thus in order to estimate the solution $u$ to (\ref{dirnono}) we can deal separately with the two integrals
$$
\displaystyle\int_0^t e^{i(t-s)|D|}F(s,x)ds\quad\textrm{and}\quad
\int_0^t\frac{e^{i(t-s)|D|}}{|D|}\mathcal{D}F(s,x)ds.
$$
We prove the following:

\begin{proposition}\label{prop2}
Let $n=3$ and assume $F(t,x)$ has the structure
\begin{equation}\label{Fform}F(t,x)=F_1(|x|)\mathbb{I}_4+i\beta(\alpha\cdot\hat{x})F_2(|x|).
\end{equation}
Then the following estimate holds
\begin{equation}\label{stdir}
\displaystyle\left\|\int_0^t e^{i(t-s)\mathcal{D}}F(s)\:ds\right\|_{L^2_tL^\infty_x}\lesssim
\| \langle x\rangle^{\frac{1}{2}+}|D| F\|_{L^2_tL^2_x}.
\end{equation}
\end{proposition}

The key
step in the proof of (\ref{stdir}) is the following 
non homogeneous estimate for the wave propagator with a radial term.

\begin{lemma}\label{nonomwave}
Let $n\geq3$, $F(t,\cdot)$ be a radial function. Then the following estimate holds
\begin{equation}\label{st1}
\displaystyle\left\|\int_0^t e^{i(t-s)|D|}F(s)\:ds\right\|_{L^2_tL^\infty_x}\lesssim
\| \langle x\rangle^{\frac{1}{2}+}|D|^\frac{n-1}{2} F\|_{L^2tL^2_x}.
\end{equation}
\end{lemma}

\begin{proof}
We start with (\ref{st1}). Expanding $u$ as in the homogeneous case (see formulas (\ref{radialw} and (\ref{max})), from the radiality of $F$ we can estimate the $L^\infty$ norm of the solution at fixed $t$ as (here  $\widehat G(s,\lambda)=\lambda^{n-1}\widehat F(s,\lambda)H(\lambda)$ and $H$ is the Heaviside function)
$$
\displaystyle\| u\|_{L^\infty_x}\lesssim\sup_r\frac{1}{r}\int_{-r}^r\left(\int_0^t\left|\widehat G(s,y+t-s)\right|ds\right)dy=
$$
$$
=\sup_r\frac{1}{r}\int_{-r}^r\left(\int_0^t\left|\widehat G(s,y+t-s)\right|\langle y+t-s\rangle^{\frac{1}{2}+}\langle y+t-s\rangle^{-\frac{1}{2}-}ds\right)dy\lesssim
$$
$$
\lesssim\displaystyle\sup_r\frac{1}{r}\int_{-r}^rdy\left[\left(\int_\mathbb{R}\left|\widehat G_{1}(s,y+t-s)\right|^2ds\right)^\frac{1}{2}
\cdot\left(\int_0^t \langle y+t-s\rangle^{-1-}ds\right)^\frac{1}{2}
\right]
$$
where in the last inequality we have used Cauchy-Schwarz inequality 
and $G_{1}$ is the function defined by
$$
\widehat G_{1}(s,y)=\widehat G(s,y)\langle y\rangle^{\frac{1}{2}+}.
$$
Setting now $h(z):=\displaystyle\left(\int_\mathbb{R}\left|\widehat G_{1}(s,z-s)\right|^2ds\right)^\frac{1}{2}$ we have 
$$
\displaystyle\sup_r\frac{1}{r}\int_{-r}^r dy\left(\int_\mathbb{R}\left|\widehat G_{1}(s,y+t-s)\right|^2ds\right)^\frac{1}{2}=
M(h)(t)
$$
The $L^p$ boundedness of the maximal operator yields
$$\
\| u\|_{L^2_tL^\infty_x}\lesssim\| h(t)\|_{L^2_t}=
\displaystyle\left(\int\int\left|\widehat G_{1}(s,t-s)^2\right|^2dsdt\right)^\frac{1}{2}=
\| \widehat G_{1}\|_{L^2_sL^2_y}
$$
The last quantity is precisely
$$
\|\langle y\rangle^{\frac{1}{2}+}\mathcal{F}_{\lambda\rightarrow y}\left(
\lambda^{n-1}\widehat F(s,\lambda)H(\lambda)\right)\|_{L^2_sL^2_y}
$$
and to conclude the proof we need to estimate it by
\begin{equation*}
\lesssim
\|\bra{x}^{\frac12+}|D|^\frac{n-1}{2}F\|_{L^{2}_{s}L^{2}_{x}}.
\end{equation*}
Since $\mathcal{F}^{-1}(|D|^\frac{n-1}{2}f)=|\xi|^\frac{n-1}{2}\check{f}$ 
we see that it is enough to prove the general inequality 
(we can neglect the dependence on time)
\begin{equation}\label{interpol}
\|\langle\rho\rangle^k\mathcal{F}_{\lambda\rightarrow\rho}\left(
\lambda^{\frac{n-1}{2}}\widehat f(\lambda) H(\lambda)\right)\|_{L^2_\rho}\lesssim
\|\langle x\rangle^k f\|_{L^2}
\end{equation}
for $k=1/2+$. 

We prove (\ref{interpol}) by interpolation. 
The case $k=0$ is trivial, since from Placherel's Theorem we obviously have
\begin{equation}\label{interpol1}
\|\mathcal{F}_{\lambda\rightarrow\rho}\left(
\lambda^{\frac{n-1}{2}}\widehat f(\lambda)H(\lambda)\right)\|_{L^2}
\cong\|\lambda^{\frac{n-1}{2}}\widehat f(\lambda)\|_{L^2}=\| f\|_{L^2}.
\end{equation}
The case $k=1$ is just a little more complicated. Since of course $\langle\rho\rangle\le 1+|\rho|$, we need only prove that
$$
\| \rho \:\mathcal{F}_{\lambda\rightarrow\rho}\left(
\lambda^{\frac{n-1}{2}}\widehat f(\lambda)H(\lambda)\right)\|_{L^2}\lesssim
\|\langle x\rangle f\|_{L^2}
$$
or equivalently
\begin{equation}
\| \partial_\lambda\left(
\lambda^{\frac{n-1}{2}}\widehat f(\lambda)H(\lambda)\right)\|_{L^2}\lesssim
\|\langle x\rangle f\|_{L^2}.
\end{equation}
We write
$$
\displaystyle\| \partial_\lambda\left(
\lambda^{\frac{n-1}{2}}\widehat f(\lambda)\chi_{\mathbb{R}^+}(\lambda)\right)\|_{L^2}\lesssim
\|\lambda^\frac{n-1}{2}\partial_\lambda\hat f(\lambda)\|_{L^2_+}+\|
\lambda^{\frac{n-3}{2}}\hat f(\lambda)\|_{L^2_+}=I_1+I_2
$$
with the shorthand notation $L^2_+=L^2(0,\infty)$.
For $I_1$ we trivially have
$$
I_1=\|\lambda^\frac{n-1}{2}\widehat{(\rho f)}\|_{L^2}\cong\| |x|f\|_{L^2}.
$$
Let's now turn to $I_2$. We split the norm
$$
\displaystyle\|
\lambda^{\frac{n-3}{2}}\hat f(\lambda)\|_{L^2_+}=
\|
\lambda^{\frac{n-3}{2}}\hat f(\lambda)\|_{L^2(\{\lambda\geq1\})}
+
\|
\lambda^{\frac{n-3}{2}}\hat f(\lambda)\|_{L^2(\{\lambda<1\})}.
$$
Plancherel's Theorem yields again for the first term
$$
\|
\lambda^{\frac{n-3}{2}}\hat f(\lambda)\|_{L^2(\{\lambda\geq1\})}\leq
\|\lambda^\frac{n-1}{2}\hat f\|_{L^2(\mathbb{R})}\cong\| f \|_{L^2}.
$$
To handle the second term we use Hardy's inequality:
$$
\|
\lambda^{\frac{n-3}{2}}\hat f(\lambda)\|_{L^2(\{\lambda<1\})}=
\||\xi|^{-1}\hat f\|_{L^2(\{|\xi|<1\})}\lesssim
\|\nabla_\xi\hat f\|_{L^2}\simeq\||x|f\|_{L^2}
$$
and this concludes the case $k=1$. 
By interpolation with (\ref{interpol1}) we obtain
the desired estimate  (\ref{st1}).
\end{proof}

\begin{lemma}\label{nonomowave2}
Let $n=3$ and $F(t,\cdot)$ be of the form (\ref{Fform}). 
Then the following estimate holds
\begin{equation}\label{st2}
\displaystyle\left\|\int_0^t \frac{e^{i(t-s)|D|}}{|D|^k}\mathcal{D}^kF(s)\:ds\right\|_{L^2_tL^\infty_x}\lesssim
\| \langle x\rangle^{\frac{1}{2}+}|D| F\|_{L^2tL^2_x}.
\end{equation}
for $k=0,1$.
\end{lemma}

\begin{proof}
Since the operator $i\mathcal{D}\beta(\alpha\cdot\hat{x})\phi$ 
applied to a radial function $\phi$ produces the radial
function $i\beta\phi''$, Lemma (\ref{nonomwave}) holds, 
and we only need to control terms of the
form (where $F_{rad}$ denotes a radial function)
$$
\left\|\int_0^t \frac{e^{i(t-s)|D|}}{|D|^{1-j}}A_jF_{rad}(s)\:ds\right\|_{L^2_tL^\infty_x}
$$
with $A_j=i\beta^j(\alpha\cdot\hat{x})$, $j=0,1$.
Recalling (\ref{repDF})-(\ref{max2}), since the quantities
$A_j$ are obviously bounded, we can estimate in both cases with
$$
  \left|\int_0^t \frac{e^{i(t-s)|D|}}{|D|^{1-j}}A_jF_{rad}(s)ds\right|
  \lesssim
  \int_0^t\frac{1}{r}\int_{-r}^r|\widehat G(s,y+t-s)|\:dsdy
$$
where as before $\widehat G(s,\lambda)=\lambda^{2}\widehat F_{rad}(s,\lambda)H(\lambda)$.
Proceeding exactly as in the proof of Lemma (\ref{nonomwave}) we obtain estimate (\ref{st2}).
\end{proof}

Estimate (\ref{stdir}) is an immediate consequence of 
(\ref{st1}), (\ref{st2}) and representation (\ref{nonomrep})

\section{Proof of Theorem (\ref{teo1})}\label{sec:omo}

We now turn to the proof of Theorem (\ref{teo1}). 
This is based on a simple application of a smoothing estimate
for a Dirac equation with potential proved in \cite{DanconaFanelli08-a}
(see also \cite{DanconaFanelli07-a}, \cite{DanconaFanelli06-a}
for related results):

\begin{theorem}[\cite{DanconaFanelli08-a}]\label{danfan}
Let $V(x)=V(x)^*$ be a $4\times4$ complex valued matrix and
assume that for some $\sigma>1$ and some sufficiently
small $\delta>0$ one has
\begin{equation}\label{Vhp2}
|V(x)|\leq\frac{\delta}{w_\sigma(x)}\qquad
\text{where}\quad
w_\sigma(x)=|x|(1+|\log|x||)^\sigma
\end{equation}
Then the following smoothing estimate holds:
\begin{equation}\label{smooth}
\| w_\sigma^{-1/2}e^{it(\mathcal{D}+V)}f\|_{L^2_tL^2_x}\lesssim\| f\|_{L^2}.
\end{equation}
\end{theorem}

It is not difficult to deduce the endpoint estimate (\ref{endV}) 
for the perturbed flow from the previous result
and our mixed endpoint-smoothing estimate \eqref{stdir}.
First of all, the solution of the equation
\begin{equation*}
  iu_{t}=\mathcal{D}u+Vu,\qquad u(0,x)=f
\end{equation*}
can be written, regarding $Vu$ as a right-hand member of the
equation
\begin{equation*}
  u=e^{it(\mathcal{D}+V)}f=
  e^{it \mathcal{D}}f+
  i\int_{0}^{t}e^{i(t-s)\mathcal{D}}(Vu)ds.
\end{equation*}
Then we can write
\begin{equation}\label{ap}
\||D|^{-1}u \|_{L^2_tL^\infty_x}=
\||D|^{-1}e^{it(\mathcal{D}+V)}f \|_{L^2_tL^\infty_x}\leq
\end{equation}
$$
\leq\left\| |D|^{-1}e^{it\mathcal{D}}f\right\|_{L^2_tL^\infty_x}+
\left\|\int_0^t 
\frac{e^{i(t-s)\mathcal{D}}}{|D|}
(V(s)e^{is(\mathcal{D}+V)}f)ds\right\|_{L^2_tL^\infty_x}.$$
The first term can be estimated by \eqref{homdir} (notice that
$|D|$ commutes with $\mathcal{D}$ and hence with the flow).
In order to apply estimate (\ref{stdir}) to the
second term we need the following

\begin{lemma}\label{stabflow}
If $f\in\mathcal{\dot{H}}^1$, then $e^{it\mathcal{D}}f\in \mathcal{\dot{H}}^1$, and if $V$ is of the form (\ref{pppot}), then $Ve^{it\mathcal{D}}f$ is of the form \eqref{Fform}.
\end{lemma}
\begin{proof}
We write $f=f_1+\mathcal{D}f_2$ with $f_1$, $f_2$ radial functions. Then we have, from (\ref{dirsol}),
$$
e^{it\mathcal{D}}f=\left(\cos(t|D|)+\frac{\sin(t|D|)}{|D|}\mathcal{D}\right)(f_1+\mathcal{D}f_2)=
$$
$$
=\cos(t|D|)f_1+\sin(t|D|)|D|f_2+\mathcal{D}\left(\cos(t|D|)f_2+\frac{\sin(t|D|)}{|D|}f_1\right)=
$$
$$
=\tilde{f}_1+\mathcal{D}\tilde{f}_2,
$$
where $\tilde{f}_1$ and $\tilde{f}_2$ are radial functions with the
appropirate regularity, 
and this concludes the proof of the first statement.
The proof of the second statement is trivial.
\end{proof}

We thus  can estimate  (\ref{ap}) with (\ref{homdir}) and (\ref{stdir}) obtaining
$$
\lesssim\|f\|_{L^2}+\|\langle x\rangle^{\frac{1}{2}+}Vu\|_{L^2_tL^2_x}
$$
Now multiplying and dividing by $w_\sigma(x)^{1/2}$ in the second norm on the right hand side yields
$$
\leq\| f\|_{L^2}+\|\langle x\rangle^{1/2+}w_\sigma^{1/2}V\|_{L^\infty}\cdot\| w_\sigma^{-1/2}u\|_{L^2_tL^2_x}.
$$
Notice that the weighted norm of $V$ at the right hand side
is bounded as it follows from assumption \eqref{Vhp}. Moreover
\eqref{Vhp} implies also that the assumption of Theorem
\ref{danfan} is satisfied.
Then using \eqref{smooth} we conclude
\begin{equation}
\| |D|^{-1}u\|_{L^2_tL^\infty_x}\leq\left(1+\|\langle x\rangle^{1/2+}w_\sigma^{1/2}V\|_{L^\infty}\right)\| f\|_{L^2}
\end{equation}
that gives, under hypothesis (\ref{Vhp}) on the potential, the desired Strichartz endpoint estimate.

\section{Partial wave subspaces and radial Dirac operator}\label{thaller} 

The purpose of this section is to construct, following \cite{Thaller92-a}, invariant subspaces for the Dirac operator with a potential having
a special symmetry. To this end we use the classical
decomposition of the space $L^2(\mathbb{R}^3)^4$ in the direct sum of 2-dimensional Hilbert spaces, the \emph{partial wave subspaces},
which are invariant for the Dirac operator. 
We shall aslo check that the lowest order partial wave subspaces are
invariant even for the cubic nonlinearities that we consider
here. 

We begin by recalling the basic facts on the decomposition,
referring to \cite{Thaller92-a} for more details. We shall use
the standard notation for polar coordinates in $\mathbb{R}^{3}$
\begin{equation*}
  x=r\sin\theta\cos\phi,\qquad
  y=r\sin\theta\sin\phi,\qquad
  z=r\cos\theta
\end{equation*}
with the unit vectors in the directions of the polar coordinate lines given by
$$
\displaystyle
\begin{cases}
e_r=(\sin\theta\cos\phi,\sin\theta\sin\phi,\cos\theta)=\displaystyle\frac{x}{|x|}=\hat{x}\\
e_\theta=(\cos\theta\cos\phi,\cos\theta\sin\phi,-\sin\theta)=
\displaystyle\frac{\partial e_r}{\partial\theta}\\
e_\phi=(-\sin\phi,\cos\phi,0)=\displaystyle\frac{1}{\sin\theta}
\frac{\partial e_r}{\partial\phi}.
\end{cases}
$$
Then we write for a function $\psi\in L^2(\mathbb{R}^3)$
\begin{equation}\label{rad}
\psi(r,\theta,\phi)=r\tilde{\psi}(x(r,\theta,\phi),y(r,\theta,\phi),z(r,\theta,\phi)).
\end{equation}
Since the function $\psi(r,\cdot,\cdot)$ of the angular variables is square integrable on the unit sphere $L^2(S^2)$, the mapping $\tilde\psi\rightarrow\psi$ defines a unitary isomorphism
$$
L^2(\mathbb{R}^3)\cong L^2((0,\infty), dr;L^2(S^2))=L^2((0,\infty),dr)\otimes L^2(S^2).
$$
Applying the transformation (\ref{rad}) on each component of the 
(vector valued)
wavefunction, we obtain the analogous decomposition
$$
L^2(\mathbb{R}^3)^4\cong L^2((0,\infty),dr)\otimes L^2(S^2)^4.
$$
The decomposition of the Hilbert space into a "radial" and an "angular" part is very useful since the angular momentum operators
$$
\textbf{L}=x\wedge (-i\nabla)\qquad \textrm{orbital angular momentum}
$$
$$
\textbf{J}=\textbf{L}+\textbf{S}\qquad\textrm{total angular momentum}
$$
act only on the angular part $L^2(S^2)^4$ in a nontrivial way;
here
\begin{equation*}
  \textbf{S}=-1/4(\alpha\wedge\alpha)
\end{equation*}
is the spin angular momentum operator.
Recalling the expression of $\nabla$ in polar coordinates
\begin{equation}\label{nabla}
\nabla=\displaystyle e_r\frac{\partial}{\partial r}+
\frac{1}{r}\left(e_\theta\frac{\partial}{\partial\theta}+e_\phi\frac{1}{\sin\theta}\frac
{\partial}{\partial\theta}\right)
\end{equation}
we obtain that, since $x=r\cdot e_r$,
\begin{equation}\label{L}
\textbf{L}=\displaystyle  i\:e_\theta\frac{1}{\sin\theta}\frac{\partial}{\partial\phi}-i\:
e_\phi\frac{\partial}{\partial\theta}
\end{equation}
where the differentiation applies to each component of the wavefunction.

The Dirac operator can be written in polar coordinates
as follows. Combining (\ref{nabla}) and (\ref{L}) yields
$$
\displaystyle -i\nabla=-i\:e_r\frac
{\partial}{\partial r}-\frac{1}{r}(e_r\wedge\textbf{L})
$$
and thus
\begin{equation}\label{dirrad1}
\displaystyle -i\alpha\cdot\nabla=-i(\alpha\cdot e_r)\frac{\partial}{\partial r}-
\frac{1}{r}\alpha\cdot(e_r\wedge \textbf{L}).
\end{equation}
By the basic property of the Dirac matrices:
$$
(\alpha\cdot A)(\alpha\cdot B)=A\cdot B+2i\textbf{S}\cdot(A\wedge B);
$$
which holds for any matrix-valued vector fields $A=(A^1,A^2,A^3)$, $B=(B^1,B^2,B^3)$ with $F^i,G^i\in \mathcal{M}_{4\times4}(\mathbb{C})$, and
$$
\gamma_5\alpha=2\textbf{S},
$$
where $\gamma_5=\left(\begin{array}{cc}0 & 1 \\1 & 0\end{array}\right),$
we obtain
$$
(\alpha\cdot A)(2\textbf{S}\cdot B)=i\gamma_5 A\cdot B-i\alpha\cdot(A\wedge B),
$$
thus equation (\ref{dirrad1}) is equal to
$$
=-i(\alpha\cdot e_r)\frac{\partial}{\partial r}+\frac{i}{r}
(\alpha\cdot e_r)(2\textbf{S}\cdot\textbf{L}).
$$
Finally, introducing the \emph{spin orbit operator}
\begin{equation}\label{spinorb}
  K=\beta(2\textbf{S}\cdot\textbf{L}+1)\equiv \beta(J^2-L^2+1/4)
\end{equation}
where we used the identity $J^2=(L+S)^2=L^2+2S\cdot L+3/4$,
we arrive at the following representation:

\begin{proposition}\label{diracrep}
The 3-dimensional Dirac operator can be written as
\begin{equation}\label{dirfor}
\mathcal{D}=-i(\alpha\cdot\hat{x})
\left(\frac{\partial}{\partial r}+
\frac{1}{r}-\frac{1}{r}\beta K\right)
\end{equation}
where $K$ is the spin orbit operator defined in (\ref{spinorb}).
\end{proposition}

The key step to construct the invariant spaces is the following:

\begin{proposition}\label{eigen}
  For each choice $(j,m_{j},k_{j})$ with
  $j=\frac{1}{2},\frac{3}{2},\frac{5}{2},..$, 
  $m_j=-j,-j+1,...,+j$, $k_j=-(j+1/2), +(j+1/2)$, there exist precisely
  two eigenfunctions $\Phi^{\pm}_{m_j,k_j}\in C^\infty(S^2)^4$
  satisfying the following relations:
  $$
  J^2\Phi_{m_j,k_j}=j(j+1)\Phi_{m_j,k_j},
  $$
  $$
  J_3\Phi_{m_j,k_j}=m_j\Phi_{m_j,k_j},
  $$
  $$
  K\Phi_{m_j,k_j}=-k_j\Phi_{m_j,k_j}.
  $$
  The family $\Phi^{\pm}_{m_j,k_j}$ forms an orthonomral basis of 
  $L^{2}(\mathbb{S}^{2})^{4}$.
\end{proposition}

The functions $\Phi_{m_j,k_j}$ can be written
explicitly using spherical harmonics.
We first recall the following representation of 3-dimensional spherical harmonics
\begin{equation}\label{harm}
Y^m_l(\theta,\phi)=\displaystyle\sqrt{\frac{2l+1}{4\pi}\frac{(l-m)!}{(l+m)!}}e^{im\phi}
P^m_l(\cos\theta)\qquad\forall-l\leq m\leq l.
\end{equation}
where $P^m_l$ are the Legendre polynomials
\begin{equation}\label{leg}
P^m_l(x)=
\displaystyle\frac{(-1)^m}{2^ll!}(1-x^2)^{m/2}\frac{d^{m+l}}{dx^{m+l}}(x^2-1)^l.
\end{equation}
As it is well known, the spherical harmonics form a complete orthonormal set in $L^2(S^2)$, i.e. every function $f\in L^2(S^2)$ can be written as
$$
f(\theta,\phi)=\displaystyle\sum_{l=0}^\infty
\sum_{m=-l}^lf_l^mY^m_l(\theta,\phi)
$$
for some constants $f^m_l$; moreover, they are eigenfunctions of both the operators $L^2$ and $L_3$, i.e.
\begin{equation}\label{L2}
L^2Y^m_l=l(l+1)Y^m_l
\end{equation}
\begin{equation}\label{L3}
L_3Y^m_l=mY^m_l.
\end{equation}
We now define for $j=\frac{1}{2},\frac{3}{2},\frac{5}{2},..$, $m_j=-j,-j+1,...,+j$ the functions $\Psi^{m_j}_{j\mp 1/2}\in L^2(S^2)^2$:
\begin{equation}\label{psi-}
\Psi^{m_j}_{j-1/2}=\displaystyle\frac{1}{\sqrt{2j}}
\left(\begin{array}{cc}\sqrt{j+m_j}\:Y^{m_j-1/2}_{j-1/2}\\
\sqrt{j-m_j}\:Y^{m_j+1/2}_{j-1/2}
\end{array}\right)
\end{equation}
\begin{equation}\label{psi+}
\Psi^{m_j}_{j+1/2}=\displaystyle\frac{1}{\sqrt{2j+2}}
\left(\begin{array}{cc}\sqrt{j+1-m_j}\:Y^{m_j-1/2}_{j+1/2}\\
-\sqrt{j+1+m_j}\:Y^{m_j+1/2}_{j+1/2}
\end{array}\right).
\end{equation}
These functions are, as it is easily seen, eigenfunctions of both  the operators $L^2$ and $J^2=L^2+\sigma\cdot L+3/4$  with eigenvalues 
$l(l+1)$ and $j(j+1)$ respectively.
So we conclude, in view of (\ref{spinorb}), that the functions 
in Proposition (\ref{eigen}) are given by
\begin{equation}\label{phi}
\Phi^+_{m_j,\mp(j+1/2)}=\displaystyle
\left(\begin{array}{cc}i\Psi^{m_j}_{j\mp1/2}\\
0
\end{array}\right)\qquad\Phi^-_{m_j,\mp(j+1/2)}=\displaystyle
\left(\begin{array}{cc}0\\
\Psi^{m_j}_{j\pm1/2}.
\end{array}\right).
\end{equation}
Thus the Hilbert space $L^2(S^2)^4$ is the orthogonal direct sum of 2-dimensional Hilbert spaces $\mathcal{H}_{m_j,k_j}$, which are spanned by simultaneous eigenfunctions $\Phi^\pm_{m_j,k_j}$ of $J^2$ and $K$:
\begin{equation}\label{dirsum}
L^2(S^2)^4=\displaystyle\bigoplus_{j=\frac{1}{2},\frac{3}{2},...}^\infty
\bigoplus_{m_j=-j}^j\bigoplus_{k_j=\pm(j+\frac{1}{2})}\mathcal{H}_{m_j,k_j}
\end{equation}
Easy calculations show that the functions $\Psi^{m_j}_{j\pm1/2}$ satisfy
$$
\displaystyle(\sigma\cdot\hat{x})\Psi^{m_j}_{j\pm1/2}=\Psi^{m_j}_{j\mp1/2},
$$
and hence
\begin{equation}\label{alphasig}
i(\alpha\cdot\hat{x})\Phi^\pm_{m_j,k_j}=\mp\Phi^\mp_{m_j,k_j}.
\end{equation}
This proves the following:

\begin{lemma}\label{stab}
The subspaces $\mathcal{H}_{m_j,k_j}$ are left invariant by the operators $\beta$ and\\ $\alpha\cdot\hat{x}$. With respect to the basis $\{\Phi_{m_j,k_j}^+\Phi_{m_j,k_j}^-\}$ defined above, these operators are represented by the $2\times 2$ matrices
\begin{equation}\label{betaalpha}
\beta=\left(\begin{array}{cc}1 & 0\\0 & -1\end{array}\right),\qquad
-i\alpha\cdot\hat{x}=\left(\begin{array}{cc}0 & -1\\1 & 0\end{array}\right).
\end{equation}
\end{lemma}

The decomposition just shown obviously implies a similar one of $L^2(\mathbb{R}^3)^4$, in which each \emph{partial wave subspace} $L^2((0,\infty),dr)\otimes\mathcal{H}_{m_j,k_j}$ is isomorphic to $L^2((0,\infty),dr)^2$ if we choose the basis $\{\Phi^+_{m_j,k_j},\Phi^-_{m_j,k_j}\}$. There is in fact a unitary isomorphism between the Hilbert spaces:
\begin{equation}\label{isom}
L^2(\mathbb{R}^3)^4\cong \bigoplus L^2((0,\infty),dr)\otimes \mathcal{H}_{m_j,k_j}.
\end{equation}
This decomposition and (\ref{dirfor}) allow us to easily calculate the action of the Dirac operator (at least on differentiable states) even in the presence of a suitable potential.

\begin{proposition}\label{dir+v}
The Dirac operator (\ref{dirfor}) with the potential
\begin{equation}\label{V}
V(x)=V_1(|x|)\mathbb{I}_4+i\beta(\alpha\cdot\hat{x})V_2(|x|)
\end{equation}
leaves the partial wave subspaces $C^\infty_0(0,\infty)\otimes \mathcal{H}_{m_j,k_j}$ invariant. With respect to the basis ${\Phi_{m_j,k_j}^+,\Phi^-_{m_j,k_j}}$ the Dirac operator on each subspace can be represented by the operator
\begin{equation}\label{action}
\displaystyle
d_{m_j,k_j}=\left(\begin{array}{cc}V_1(|x|) & -\frac{d}{dr}+\frac{k_j}{r}
+V_2(|x|)\\  \frac{d}{dr}+\frac{k_j}{r}
+V_2(|x|) & V_1(|x|)\end{array}\right)
\end{equation}
which is well defined over $C^\infty_0(0,\infty)^2\subset L^2((0,\infty),dr)^2$.
Moreover, the Dirac operator $\mathcal{D}$ on $C^\infty_0(\mathbb{R}^3)^4$is unitary equivalent to the direct sum of the partial wave Dirac operators $d_{m_j,k_j}$,
\begin{equation}\label{pardir}
\mathcal{D}\cong \displaystyle\bigoplus_{j=\frac{1}{2},\frac{3}{2},...}^\infty
\bigoplus_{m_j=-j}^j\bigoplus_{k_j=\pm(j+\frac{1}{2})}d_{m_j,k_j}\end{equation}
\end{proposition}

\begin{remark}\label{rem}
Proposition \ref{dir+v} holds for slightly more general potentials (see \cite{Thaller92-a}), but we shall not need this fact here.
\end{remark}

\begin{remark}\label{rkrad}
The operator in (\ref{action}) is also known as \emph{radial Dirac operator}. It can be proved that $d_{m_j,k_j}$ is essentially self-adjoint (for every $j$) on $C^\infty_0(0,\infty)$ if and only if $\mathcal{D}+V$ is essentially self-adjoint on $C^\infty_0(\mathbb{R}^3\backslash \{0\})$.
\end{remark}

Thus using spherical coordinates it is possible to construct invariant spaces for the perturbed Dirac operator. What may come as a surprise is that for $j=1/2$ the partial wave subspaces are also invariant for the cubic nonlinearity, and this fact is obviously crucial for the nonlinear application we shall prove in the next section.
\begin{lemma}\label{1/2}
Let $j=1/2$ and let $(m_{1/2}, k_{1/2})$ be one of the couples
(-1/2,-1), (-1/2,1), (1/2,-1), (1/2,1). Then the partial wave subspaces $C^\infty_0((0,\infty),dr)\otimes  \mathcal{H}_{m_{1/2},k_{1/2}}$ are invariant for the cubic nonlinearities $F_1(u)=\langle u,u\rangle u$ and $F_2(u)=\langle \beta u,u\rangle u$, i.e. \begin{equation}\label{stab}
u\in C^\infty_0((0,\infty),dr)\otimes  \mathcal{H}_{m_{1/2},k_{1/2}}\Rightarrow F_1(u),\:F_2(u)\in C^\infty_0((0,\infty),dr)\otimes  \mathcal{H}_{m_{1/2},k_{1/2}}.
\end{equation}
\end{lemma}\begin{proof}
We explicitly write down the functions $\Phi^+$, $\Phi^-$ in the four cases: a straightforward calculation using formulas (\ref{harm}), (\ref{leg}), (\ref{psi-}), (\ref{psi+}) and (\ref{phi}) yields
\begin{equation}\label{phi1}
\Phi^+_{-1/2,-1}=
\left(\begin{array}
{cc}
0\\ \displaystyle\frac{i}{2\sqrt\pi}\\
0\\0
\end{array}\right)\qquad
\Phi^-_{-1/2,-1}=
\left(\begin{array}
{cc}
0\\0\\
\displaystyle\frac{1}{2\sqrt\pi}e^{i\phi}\sin\theta\\
-\displaystyle\frac{1}{2\sqrt\pi}\cos\theta
\end{array}\right).
\end{equation}

\begin{equation}\label{phi2}
\Phi^+_{-1/2,1}=
\left(\begin{array}
{cc}
\displaystyle\frac{i}{2\sqrt\pi}e^{i\phi}\sin\theta\\
\displaystyle
-\frac{i}{2\sqrt\pi}\cos\theta\\
0\\0
\end{array}\right)\qquad
\Phi^-_{-1/2,1}=
\left(\begin{array}
{cc}
0\\0\\
0\\\displaystyle\frac{1}{2\sqrt\pi}
\end{array}\right).
\end{equation}

\begin{equation}\label{phi3}
\Phi^+_{1/2,-1}=
\left(\begin{array}
{cc}
\displaystyle\frac{i}{2\sqrt\pi}\\
0 \\
0\\0
\end{array}\right)\qquad
\Phi^-_{1/2,-1}=
\left(\begin{array}
{cc}
0\\0\\
\displaystyle\frac{1}{2\sqrt\pi}\cos\theta\\
\displaystyle\frac{1}{2\sqrt\pi}e^{i\phi}\sin\theta
\end{array}\right).
\end{equation}

\begin{equation}\label{phi4}
\Phi^+_{1/2,1}=
\left(\begin{array}
{cc}
\displaystyle\frac{i}{2\sqrt\pi}\cos\theta\\
\displaystyle
\frac{i}{2\sqrt\pi}e^{i\phi}\sin\theta\\
0\\0
\end{array}\right)\qquad
\Phi^-_{1/2,1}=
\left(\begin{array}
{cc}
0\\0\\
\displaystyle\frac{1}{2\sqrt\pi}\\0
\end{array}\right).
\end{equation}
We prove Lemma (\ref{1/2}) for the couple $(1/2,1)$, i.e. for functions of the form (\ref{phi4}), being the proof for the other cases completely analogous. 

The generic function $u\in L^2((0,\infty),dr)\otimes  \mathcal{H}_{1/2,1}$ can be written as
$$
u(r,\theta,\phi)=u^+(r)\Phi^+_{1/2,1}(\theta,\phi)+u^-(r)\Phi^-_{1/2,1}(\theta,\phi)
$$
for some radial functions $u^+$, $u^-$. So $f$ takes the vectorial form
\begin{equation}\label{generic}
u=\left(\begin{array}
{cc}
\displaystyle u^+(r)\frac{i}{2\sqrt\pi}\cos\theta\\
\displaystyle
u^+(r)\frac{i}{2\sqrt\pi}e^{i\phi}\sin\theta\\
\displaystyle\frac{u^-(r)}{2\sqrt\pi}\\0
\end{array}\right).
\end{equation}
Thus the Hermitian product $\langle u,u\rangle$ yields
$$
\langle u,u\rangle=\displaystyle-\frac{1}{4\pi}\cos^2\theta\: u^+(r)^2
-\frac{1}{4\pi}\sin^2\theta\: u^+(r)^2+\frac{1}{4\pi}u^-(r)^2=
$$
$$
=-\frac{1}{4\pi}\left(u^+(r)^2-u^-(r)^2\right)
$$
that has no angular components. This proves that if $u\in C^\infty_0((0,\infty),dr)\otimes 
\mathcal{H}_{m_{1/2},k_{1/2}}$ then $F_1(u)\in C^\infty_0((0,\infty),dr)\otimes  \mathcal{H}_{m_{1/2},k_{1/2}}$.

Minor modifications yield the same result also for the nonlinearity $F_2(u)$. In fact we know from Lemma (\ref{stab}) 
that the operator $\beta$ acts on the partial wave subspaces in a very simple way with respect to the basis $\{\Phi^+,\Phi^-\}$: if in fact we associate to the function $u$ its coordinates $(u^+(r),u^-(r))$ with respect to such a basis we have $\beta u=(u^+(r),-u^-(r))$, so that
$$
\langle \beta u,u\rangle=\displaystyle-\frac{1}{4\pi}\cos^2\theta\: u^+(r)^2
-\frac{1}{4\pi}\sin^2\theta\: u^+(r)^2-\frac{1}{4\pi}u^-(r)^2=
$$
$$
=-\frac{1}{4\pi}\left(u^+(r)^2+u^-(r)^2\right)
$$
that again has no angular components, and this shows that if $u\in C^\infty_0((0,\infty),dr)\otimes 
\mathcal{H}_{m_{1/2},k_{1/2}}$ then $F_2(u)\in C^\infty_0((0,\infty),dr)\otimes  \mathcal{H}_{m_{1/2},k_{1/2}}$.
\end{proof}

\section{Global existence for the nonlinear equation}\label{sec:nl}

As an application of the results we have presented in the previous 
section we can now prove global existence for problem (\ref{nlpb})  
with small initial data in one of the four partial wave subspaces  
$\dot H^1((0,\infty),dr)\otimes  \mathcal{H}_{m_{1/2},k_{1/2}}$. 
Our goal is to prove

\begin{theorem}\label{teo2bis}
Consider the Cauchy problem for the 3-dimensional nonlinear 
Dirac equation
\begin{equation}\label{nlpb2}
iu_t-\mathcal{D}u+V(x)u=F(u),\qquad
u(0,x)=f(x)
\end{equation}
where the potential $V$ is of the form 
$$
V=V_1(|x|)\mathbb{I}_4+i\beta(\alpha\cdot\hat{x})V_2(|x|)
$$  
and satisfies assumption (\ref{Vhp}), while the nonlinear
term $F(u)$ is either of the form $\langle \beta u,u\rangle u$ or 
$\langle u,u\rangle u$.

Then for every initial data $f\in\dot H^1((0,\infty),dr)\otimes  \mathcal{H}_{m_{1/2},k_{1/2}}$, with sufficiently small $\dot H^1$ norm, 
there exists a unique global solution $u(t,x)$ to problem 
(\ref{nlpb}) in the class 
$C_t(\mathbb{R},\dot H^1)\cap L^2_t(\mathbb{R},L^\infty).$
\end{theorem}

\begin{proof}
The proof is identical for both choices of the form of the
nonlinear term.
We rewrite (\ref{nlpb}) in integral form
\begin{equation}\label{nlint}
u=\displaystyle e^{it\mathcal{D}}f+i\int_0^te^{i(t-s)\mathcal{D}}\left(V(s)u(s)+F(u(s))\right)ds=
\end{equation}
$$
=\displaystyle e^{it\mathcal{D}}f+i\int_0^te^{i(t-s)\mathcal{D}}\left(V(s)u(s)\right)ds+
i\int_0^te^{i(t-s)\mathcal{D}}\left(F(u(s))\right)ds=
$$
$$
\displaystyle =e^{it(\mathcal{D}+V)}f+i\int_0^te^{i(t-s)\mathcal{D}}\left(F(u(s))\right)ds
\equiv\Phi(u)\equiv
I_1+I_2
$$
we denote by $\Phi(u)$ the RHS of (\ref{nlint}) and we check that
the map $\Phi$ is a contraction on the function space 
\begin{equation*}
  X=L^2_tL^\infty_x\cap L^\infty_t\dot H^1_x.
\end{equation*}

In order to estimate the first term $I_1$ we use our endpoint Strichartz estimate $(\ref{endV})$, observing that if $f\in \dot H^1((0,\infty),dr)\otimes  \mathcal{H}_{m_{1/2},k_{1/2}}$ then in particular $f$ is of the form $f=f_1+\mathcal{D}f_2$ with $f_1$, $f_2$ radial functions, so estimate  (\ref{endV}) holds and gives
\begin{equation}\label{I1+I2}
\| I_1\|_X\lesssim \| f\|_{\dot H^1}
\end{equation}

Now we need to handle the nonlinear term $I_2$.  
By Minkowski inequality
\begin{equation*}
  \left\|\displaystyle\int_0^te^{i(t-s)
  \mathcal{D}}F(u(s))ds\right\|_X\leq\int_0^\infty
  \| e^{it\mathcal{D}}e^{-is\mathcal{D}}F(u(s))\|_X ds.
\end{equation*}
By energy conservation we have
\begin{equation*}
  \| e^{it\mathcal{D}}e^{-is\mathcal{D}}F(u(s))\|
  _{L^{\infty}_{t}\dot H^{1}}
  =\| F(u(s))\|_{L^{\infty}_{s}\dot H^{1}}
\end{equation*}
On the other hand, in view of Lemma (\ref{1/2}), we can use
estimate (\ref{homdir}) and we have
\begin{equation*}
  \| e^{it\mathcal{D}}e^{-is\mathcal{D}}F(u(s))\|
  _{L^{2}_{t}L^{\infty}_{x}}
  =\| F(u(s))\|_{\dot H^{1}}\le
  \| F(u(s))\|_{L^{\infty}_{s}\dot H^{1}}.
\end{equation*}
Thus
\begin{equation*}
  \|I_{2}\|_{X}\lesssim \| F(u(s))\|_{L^{\infty}_{s}\dot H^{1}}
\end{equation*}
Then by H\"older inequality in $t,x$ we obtain
$$
\| F(u)\|_{L^1_t\dot H^1_x}\leq
\| u\|_{L^2_tL^\infty_x}^2\| u\|_{L^\infty_t\dot H^1_x}
$$
which implies
\begin{equation*}
  \|\Phi(u)\|_{X}\lesssim\|f\|_{\dot H^{1}}+\|u\|_{X}^{3}.
\end{equation*}
An analogous computation gives
\begin{equation}
\| \Phi(u)-\Phi(v)\|_X\lesssim(\| u\|^2_X+\| v\|_X^2)\| u-v\|_X.
\end{equation}
Therefore if the data belong to a sufficiently small ball
in $\dot H^{1}$, $\Phi$ is a contraction on that ball, and its
unique fixed point is the unique global solutions to problem (\ref{nlpb}) 
in the space $X$.
\end{proof}

\end{document}